\documentclass[11pt, a4paper]{amsart}
\usepackage{mathrsfs}
\usepackage{amssymb}
\usepackage{enumitem}
\usepackage{esint}
\theoremstyle{plain}
\newtheorem{prop}{Proposition}
\newtheorem*{prop*}{Proposition}

\newtheorem*{thm*}{Theorem}
\newtheorem{cor}[prop]{Corollary}
\newtheorem{lem}[prop]{Lemma}
\newtheorem{question}[prop]{Question}
\newtheorem{problem}[prop]{Problem}

\newtheorem*{convention*}{Convention}
\newtheorem{thmintro}{Theorem}

\theoremstyle{definition}
\newtheorem*{defn*}{Definition}

\newtheorem{rem}[prop]{Remark}

\newtheorem*{rem*}{Remark}
\newtheorem*{rems*}{Remarks}
\newtheorem*{scholium*}{Scholium}

\newtheorem*{example*}{Example}

\newcommand{\ro}{\varrho}
\newcommand{\fhi}{\varphi}
\newcommand{\teta}{\vartheta}
\newcommand{\se}{\subseteq}
\newcommand{\lra}{\longrightarrow}

\newcommand{\inv}{^{-1}}

\newcommand{\CC}{\mathbf{C}}
\newcommand{\NN}{\mathbf{N}}
\newcommand{\PP}{\mathbf{P}}
\newcommand{\RR}{\mathbf{R}}
\newcommand{\GL}{\mathbf{GL}}

\newcommand{\SL}{\mathbf{SL}}
\newcommand{\Id}{\mathrm{Id}}
\newcommand{\Jac}{\mathrm{Jac}}
\newcommand{\Vol}{\mathrm{Vol}}
\newcommand{\sI}{\mathscr{I}}
\newcommand{\sL}{\mathscr{L}}
\newcommand{\sU}{\mathscr{U}}

\newcommand{\mes}{\sL^\infty}
\DeclareMathOperator*{\Lim}{Lim}
\newcommand{\one}{\boldsymbol{1}}
\newcommand{\acts}{\curvearrowright}

\begin{document}
\title{Equivariant measurable liftings}
\author[N. Monod]{Nicolas Monod$^\ddagger$}
\address{EPFL, 1015 Lausanne, Switzerland}
\thanks{$^\ddagger$Supported in part by the ERC}
%
%
%
\begin{abstract}
We discuss equivariance for linear liftings of measurable functions. Existence is established when a transformation group acts amenably, as e.g.\ the M\"obius group of the projective line.

Since the general proof is very simple but not explicit, we also provide a much more explicit lifting for semi-simple Lie groups acting on their Furstenberg boundary, using unrestricted Fatou convergence. This setting is relevant to $L^\infty$-cocycles for characteristic classes.
\end{abstract}
\maketitle

\section{Introduction}
\subsection{Context}
Let $f$ be a measurable function on $\RR$, in the sense of Lebesgue (real or complex). In many cases, two such functions are identified if they agree almost everywhere. That is, the actual function $f$ is forsaken for its \emph{function class} $[f]$.

It is not just a mere convenience to ignore null-sets; it is indeed unavoidable when we use certain tools of functional analysis for the Lebesgue spaces $L^p$ of $p$-summable function classes. For instance, $L^p$ is a dual Banach space for $1<p\leq \infty$ and therefore one can use weak-* compactness arguments. Such tools are not available if one insists instead to work with the spaces $\sL^p$ of actual functions.

\medskip
However, neglecting null-sets is a luxury that cannot \emph{always} be afforded; for instance, when it is needed to evaluate a function at a specific point or on a locus of interest which happens to be negligible in measure. This situation arises for cocycles representing characteristic classes~\cite{Burger-Iozzi_boundary, Bucher-Monod,Hartnick-Ott_arxI}. Von Neumann~\cite{vonNeumann31} has investigated the possibility to choose a section or ``lifting''
$$\lambda\colon L^\infty(\RR) \lra \sL^\infty(\RR)$$
of the quotient map $f\mapsto [f]$. Another instance where such a lifting is needed is for the proof of the general Dunford--Pettis theorem, as observed by Dieudonn\'e~\cite{Dieudonne51}. As to von Neumann, he mentions an unspecified operator-theoretical motivation suggested by Haar. It is time to recall the formal definition(s) of liftings:

\begin{defn*}
Let $X$ be a locally compact space endowed with a Radon measure. A \textbf{linear lifting} is a positive linear map $\lambda\colon L^\infty(X) \to \sL^\infty(X)$ such that $\lambda(\fhi)\in\fhi$ for all $\fhi\in L^\infty(X)$ and $\lambda([\one_X])=\one_X$. It is a \textbf{strong} linear lifting if $\lambda([f])=f$ whenever $f$ is continuous.
\end{defn*}

Awkwardly, a (strong) linear lifting is simply called a (strong) lifting when it is moreover multiplicative.

\begin{rems*}
(a)~A linear lifting does not increase the norms, where $\sL^\infty$ is endowed with the sup-norm and $L^\infty$ with the corresponding quotient norm: the essential sup-norm.

\smallskip
(b)~Von Neumann proved that there is no lifting for general \emph{unbounded} functions, see footnote p.~109 in~\cite{vonNeumann31}. When $p<\infty$, the spaces $L^p$ do not even admit linear liftings~\cite[Thm.~7]{IonescuTulcea62}.

\smallskip
(c)~When $X$ is a differentiable manifold, we shall always endow it with the Lebesgue measure associated to some Riemannian structure; the spaces $\sL^\infty(X)$ and $L^\infty(X)$ do not depend on the choice of the Riemannian structure, since the corresponding measures differ by a continuous density only.

\smallskip
(d)~In general, function classes in $L^\infty(X)$ are defined by identifying functions that agree \emph{locally} almost everywhere; as soon as the Radon measure is $\sigma$-finite, e.g.\ when $X$ is $\sigma$-compact, this is the usual a.e.\ identification as in the case of $\RR$.

\smallskip
(e)~Some authors, including Bourbaki's posse, let $\mes$ contain \emph{essentially} bounded functions and hence endow it with the quotient \emph{semi}-norm~\cite[IV\S6]{BourbakiINT16_anglais}. Thus our liftings provide a fortiori liftings into these larger spaces.
\end{rems*}

Von Neumann proved that there is a lifting $L^\infty(\RR) \to \sL^\infty(\RR)$.  An essential ingredient of his proof is Lebesgue's differentiation theorem (\S34--35 in ~\cite{Lebesgue10}), which has to be combined with a suitable form of the axiom of choice. Since Lebesgue differentiation is unaffected by translations, it follows that von Neumann's lifting can be chosen to commute with all translations of the line; this is particularly clear in Dieudonn\'e's account~\cite{Dieudonne51}. This use of differentiation makes it plain that more generally any Lie group will admit a lifting that commutes with left translations, and it is unsurprising that this can be extended to all locally compact groups~\cite{IonescuTulcea67}, thanks to the solution of Hilbert's fifth problem~\cite{Montgomery-Zippin} and to non-$\sigma$-finite versions of von Neumann's lifting~\cite{Maharam}, \cite{IonescuTulcea}.

\bigskip
The general question that we shall address in this paper is whether a lifting on a space $X$ can be equivariant under a transformation group $G\acts X$. Typically, we have in mind much larger groups of symmetries than just $G$ acting on $G$ itself. For instance, on the real line~$\RR$, one can consider the group of all affine transformations or even the group $\SL_2(\RR)$ acting projectively, in which case we add $\infty$ to the line.

It has been shown that multiplicative liftings cannot be equivariant in that setting~\cite[p.~95]{vonWeizsaecker}, and therefore we are asking for equivariant \emph{linear} liftings. Known results include the case of compact transformation groups~\cite{Johnson79} and more generally distal systems~\cite{Johnson80b}. Equivariant linear liftings also exists for countable amenable groups~\cite{IonescuTulcea65}. Unfortunately,  that argument does not apply to uncountable amenable groups such as the affine group of~$\RR$, nor to any non-discrete group. There is an intrinsic difficulty when the acting group is non-discrete, originating in the well-known fact that Fubini's theorem has no converse (this is one~\cite{Sierpinski20} of Sierpi\'nski's contributions to the first volume of \emph{Fundamenta} ---~in which he authored more than half the papers).

Specifically, one cannot first take a linear lifting on $X$ and then apply some averaging procedure over $G$. Indeed, the map $(g,x)\mapsto (\lambda g \fhi)(x)$ need not be measurable~\cite{Talagrand82}.

\subsection{Amenable actions}
It turns out that a general condition allowing us to construct equivariant linear liftings is the topological amenability of the \emph{action}, a much weaker property than the amenability of the group. For instance, $\SL_2(\RR)$ acts amenably on the projective line $\RR\cup\{\infty\}$. More generally, any semi-simple Lie group acts amenably on its Furstenberg boundary. We shall recall below the definition of amenable actions as found e.g.\ in~\cite{Anantharaman02, Anantharaman-Renault}. For now, we just point out that the condition is automatically satisfied for amenable groups but holds much more generally, for instance for $G\acts G/H$ where $G$ is an arbitrary locally compact group and $H<G$ a closed amenable subgroup.

\begin{thmintro}\label{thm:main}
Let $G$ be a locally compact group with a $C^1$-continuous amenable action on a differentiable manifold~$X$.

Then there exists a $G$-equivariant strong linear lifting $L^\infty(X)\to \mes(X)$.
\end{thmintro}

By definition, we say that a continuous action of a topological group $G$ on a differentiable manifold~$X$ is \textbf{$\boldsymbol C^{\boldsymbol{1}}$-continuous} if each $g\in G$ has a derivative at each $x\in X$ and the derivative depends continuously on $(g, x)$ in $G\times X$.

\medskip
Amenability is still not the optimal assumption, see Section~\ref{sec:rems-p}(a). Whether smoothness is required is unclear, see Section~\ref{sec:rems-p}(b).

\begin{rem}\label{rem:cocycles}
All assumptions of Theorem~\ref{thm:main} are preserved if we replace $X$ by a power $X^p$ with $p\in\NN$; therefore, there is a family of liftings on $L^\infty(X^p)$. We shall show in Section~\ref{sec:rems} that this family can be chosen in such a way that it moreover intertwines the homogeneous coboundary maps $L^\infty(X^{p})\to L^\infty(X^{p+1})$ and $\mes(X^{p})\to \mes(X^{p+1})$ and is compatible with permuting coordinates. This implies that bounded cohomology classes can be represented by $\mes$-cocycles, see Corollary~\ref{cor:coh} below.
\end{rem}

\begin{rem}\label{rem:euler}
If we drop the amenability assumption from Theorem~\ref{thm:main}, then \emph{we do not expect the existence of an equivariant linear lifting} in general. In the very explicit example of $\GL_m^+(\RR)$ acting on the projective space (which is non-amenable iff $m\geq 3$), a family of liftings as in Remark~\ref{rem:cocycles} cannot exist when $m$ is even with $m\geq 4$. The argument given in Section~\ref{sec:rems} is very similar to an observation in~\cite{Bucher-Monod}.
\end{rem}

\subsection{The case of semi-simple Lie groups}\label{sec:thm:Lie}
Theorem~\ref{thm:main} is ``soft'' in the sense that it is general but its very simple proof gives limited insight. In concrete situations, such as the M\"obius action of $\SL_2(\RR)$ on $\RR\cup\{\infty\}$, much more structure is available. We shall study more carefully the setting of Lie groups because one of our motivations for this paper is the lifting of cocycles representing characteristic classes. In general this amounts to lifting $L^\infty$-cocycles on a Furstenberg boundary, see e.g.~\cite{Bucher-Monod}, \cite{Hartnick-Ott_arxI}. Additional motivation to find such liftings of cocycles is provided by~\cite{Burger-Iozzi_boundary}.

\bigskip
Let $G$ be a connected (or almost connected) semi-simple Lie group. We refer to~\cite{Helgason01} and~\cite{Knapp02} for basic facts and the notation below. Choose an Iwasawa decomposition $G=KAN$ and let $M$ be the centraliser of $A$ in $K$. The \textbf{Furstenberg boundary} of $G$ can be defined as the homogeneous space $B=G/MAN\cong K/M$. We assume that $G$ has finite centre; this does not restrict the generality since the centre acts trivially on $B$ anyway. We assume that $G$ is non-compact since otherwise $B$ is trivial.

Let $X=G/K$ be the symmetric space and recall that any point $\xi\in B$ can be considered as a Weyl chamber at infinity, thus providing a family of directions towards infinity in $X$.

\medskip
We shall now construct an explicit family of \emph{truncated non-tangential domains} $V_t^\xi \se X$ with $t\geq 0$. Each $V_t^\xi$ should be thought of as a compact prism based at the origin $o=eK\in X$ and pointing towards $\xi$, extending further and further towards $\xi$ as $t\to\infty$.

\smallskip
To this end, we choose once and for all a compact set $D\se N$ with non-empty interior. We can, and do, choose $D$ to be $M$-invariant since $M$ is compact and normalises $N$. Let $\log\colon A\to\mathfrak a$ be the inverse of the exponential map, where $\mathfrak a$ is the Lie algebra of $A$. We denote by $A^+\se A$ (resp.\ $\mathfrak a^+\se \mathfrak a$) the positive Weyl chamber and by $\overline{A^+}$ its closure in $A$. Let $\ro\colon \mathfrak a\to\RR$ be the Weyl vector, i.e.\ the half-sum of positive restricted roots. We define
$$V_t^\xi \ =\ k \, A_t \,D \,o, \ \text{ where }  A_t=\big\{a\in \overline{A^+} : \ro(\log a) \leq t\big\}$$
and where $kM=\xi$ in $B\cong K/M$; notice that $V_t^\xi$ does not depend on the choice of $k\in K$ modulo $M$.

We denote by $\fint dx$ normalised integrals on $X$ with respect to the measure induced by a Haar measure on $G$. Equivalently, this is the Lebesgue measure on $X$ for the Riemannian structure induced by the Killing form of $\mathfrak g$ if we choose the right normalisation of the Haar measure. Given an integrable function class $\fhi$ on $B$, we denote by $P\fhi$ its Poisson transform (p.~100 in~\cite{Helgason08} or p.~279 in~\cite{Helgason00}), which is a harmonic function on $X$. Finally, we choose a non-principal ultrafiltre $\sU$ on $\NN$ and denote by $\Lim_{n,\sU}$ the corresponding ultralimits over $n\in\NN$.

\begin{thmintro}\label{thm:Lie}
The expression
\begin{equation}\label{eq:Lie}
(\lambda \fhi) (\xi)\ =\ \Lim_{n,\sU} \fint_{V_n^\xi} (P\fhi) (x) \,dx
\end{equation}
defines a $G$-equivariant strong linear lifting $\lambda\colon L^\infty(B)\to \sL^\infty(B)$.
\end{thmintro}

The interest of this explicit formula is that it can allow exact computations even at points that are not Lebesgue points. Indeed, if $\phi$ has some symmetries, it can happen that the integrals in~\eqref{eq:Lie} converge as $n\to \infty$. This is especially true if $\fhi$ is a simple configuration invariant, as is precisely the case when one considers $L^\infty$-cocycles on Furstenberg boundaries.

\medskip
Here is the most elementary example. Consider $G=\SL_2(\RR)$ acting on $B=\RR\cup\{\infty\}$ by M\"obius transformations. We consider the function class $\fhi$ given by the sign of $\xi \in \RR$, undefined at $0$ and $\infty$. The latter are two points where a general lifting could take unpredictable values.

In the present case, the above machinery boils down to the following classical situation: $X$ is the upper half plane in $\CC$ and the Poisson transform is easily computed to be $P\fhi (x) = 1-2\arg(x) /\pi$. Take $A^+$ to consist of the diagonal matrices $\left(\begin{smallmatrix}a&0\\ 0&1/a\end{smallmatrix}\right)$ with $a>1$; then $N$ is the group of matrices $\left(\begin{smallmatrix}1&u\\ 0&1\end{smallmatrix}\right)$. We define $D$ by $|u|\leq 1$. Then $V_t^0$ and $V_t^\infty$ are both horizontal slices of the cone $|\mathrm{Re}(x)| \leq \mathrm{Im}(x)$. Therefore, by sagittal symmetry we find $(\lambda \fhi) (\xi)=0$ for both $\xi=0, \infty$. Such symmetries are particularly useful for cocycles, compare Remark~\ref{rem:family-Lie} below.

Observe that the choice of $D$ affects the outcome; for instance, defining $D$ by $0\leq u\leq 1$ yields the value $\pi/4 - \log\sqrt 2$ instead of $0$.

\subsection{Notes on (ir)regularity}
Any \emph{individual} measurable function class admits a Borel representative, as follows e.g.\ from Lusin's theorem. It is therefore tempting to require that a lifting yield such Borel representatives. In the few cases where explicit representatives are known for $L^\infty$-cocycles representing characteristic classes on flag manifolds, these representatives are indeed Borel. In fact they loiter on the lower rungs of the hierarchies of Hausdorff--Young (\cite[IX]{Hausdorff14},\cite{Young13}) or Baire~\cite{Baire05}.

\smallskip
In the setting considered by von Neumann, it is possible to obtain a Borel lifting under the continuum hypothesis, see p.~371--372 in~\cite{vonNeumann-Stone}. Such a Borel lifting cannot, however, be equivariant under translations. This fact was discovered in~\cite{Johnson80}, extended to abelian groups in~\cite{Talagrand82}, to general locally compact groups in~\cite{Kupka-Prikry83} and even beyond~\cite{Burke2007}.

How about the (non-equivariant) von Neumann--Stone lifting \emph{without} the continuum hypothesis? Even assuming a strong negation of CH such as $2^{\aleph_0}=\aleph_2$, both the non-existence and the existence of Borel liftings are consistent with ZFC (provided ZFC is consistent). See~\cite{Shelah83} for non-existence and~\cite{Carlson-Frankiewicz-Zbierski} for existence.

\medskip
Finally, a few comments on the condition that a lifting be \emph{strong}. This condition is notably of importance for the disintegration of measures~\cite{IonescuTulcea64}. It can be satisfied e.g.\ when the topology of $X$ has a countable base~\cite{IonescuTulcea}. A slightly weaker conclusion holds for a base of cardinality $\aleph_1$~\cite{Bichteler73}, but there exists a space with a base of cardinality $\aleph_2$ that does not admit a strong lifting~\cite{Losert79}.

\section{Proof of Theorem~\ref{thm:main}}\label{sec:main}
\begin{flushright}
\begin{minipage}[t]{0.75\linewidth}\itshape\small
This new integral of Lebesgue is proving itself a wonderful tool. I might compare it to a modern Krupp gun, so easily does it penetrate barriers which before were impregnable.
\begin{flushright}
\upshape\small
\cite{VanVleck}, page~7
\end{flushright}
\end{minipage}
\end{flushright}
\vspace{1mm}
Choose a left Haar measure on $G$ and write $\int dg$ for the corresponding integrals. The amenability of the $G$-action on $X$ means that there is a net (indexed by some directed set $I$) of non-negative continuous functions $\mu_i$ on $G\times X$ such that

\begin{enumerate}[label=(\alph*)]
\item $\int_G \mu_i(g,x)\,dg =1$ for all $i\in X$ and all $x\in X$,\label{pt:defamen:norm}
\item $\lim_{i\in I} \int_G |\mu_i(sg, sx)-\mu_n(g,x)|\,dg = 0$ uniformly for $(s,x)$ in compact subsets of $G\times X$,\label{pt:defamen:asy}
\end{enumerate}
%
see Proposition~2.2~\cite{Anantharaman02}.

\medskip
We choose a Riemannian metric on $X$ and denote by $B(z,r)$ the corresponding ball of radius $r>0$ around any $z\in X$. We denote by $\int dy$ and $\fint dy$ the integrals, respectively normalised integrals with respect to the associated Lebesgue measure on $X$. Furthermore, we choose an ultrafiltre $\sI$ on $I$ dominating the order filtre, a non-principal ultrafiltre $\sU$ on $\NN$ and denote by $\Lim_{i,\sI}, \Lim_{n,\sU}$ the corresponding ultralimits. We contend that the desired lifting for a function class $\fhi=[f]$ will be provided by the expression
\begin{equation}\label{eq:lift-amen:1}
(\lambda \fhi)(x)\ =\ \Lim_{i,\sI} (\lambda_i \fhi)(x),
\end{equation}
wherein
\begin{equation}\label{eq:lift-amen:2}
(\lambda_i \fhi)(x)\ =\ \Lim_{n,\sU} \int\limits_G \fint\limits_{B(g\inv x, 1/n)} f(gy) \mu_i(g,x)  \,dy\, dg.
\end{equation}
Regarding well-posedness, we notice first that the function
\begin{equation*}\label{eq:well-posed}
(g,x) \longmapsto  \fint\limits_{B(g\inv x, 1/n)} f(gy) \,dy
\end{equation*}
depends on the class $\fhi$ of $f$ rather than on $f$, is bounded by the sup-norm $\|\fhi\|_\infty$ and is continuous on $G\times X$. Indeed, the map $X\to L^1(X)$ sending $z\in X$ to the characteristic function of $B(z, 1/n)$ is norm-continuous. In particular, the integral over $G$ and the ultralimit over $n$ both make sense in~\eqref{eq:lift-amen:2}. Next, observe that $\lambda_i$ is linear in $\fhi$, positive, and sends $[\one_X]$ to $\one_X$. We claim that $\lambda_i(\fhi)$ represents $\fhi$; this implies notably the measurability of $\lambda_i(\fhi)$. In fact we claim more: let $x$ be any Lebesgue point of $f$; we shall show $(\lambda_i \fhi)(x) = f(x)$ for all $i\in I$. The definition~\eqref{eq:lift-amen:1} then implies $(\lambda \fhi)(x) = f(x)$ and hence $\lambda$ is a lifting.

\medskip

To prove the claim, fix first $g\in G$. Let $R_n>0$ (respectively $r_n>0$) smallest (resp.\ largest) radius such that
\begin{equation}\label{eq:quasiball}
B(x, r_n) \ \se\ g B(g\inv x, 1/n)\ \se\ B(x,R_n).
\end{equation}
Since $g$ is $C^1$ at $x$, the sequence $R_n/r_n$ is bounded (depending on $g,x$). In other words, the sets in~\eqref{eq:quasiball} have bounded excentricity with respect to balls around $x$. Thus Lebesgue differentiation holds at $x$ for any $g\in G$ in the sense that
$$\lim_{n\to\infty }\fint\limits_{gB(g\inv x, 1/n)} f(y) \,dy =f(x),$$
see e.g.\ Theorem~7.10 in~\cite{RudinRCA} (or~\cite{dePossel36}). In order to deduce the claim, it remains only to show that the difference
\begin{equation}\label{eq:Jaco}
\fint\limits_{gB(g\inv x, 1/n)} f(y) \,dy  \ - \fint\limits_{B(g\inv x, 1/n)} f(gy) \,dy
\end{equation}
converges to zero uniformly over $g$ in compact subsets of $G$. Indeed, the uniformity will ensure that the convergence survives after integrating against the probability density $\mu_i(g,x) dg$. Denoting by $\Vol(\cdot)$ the normalising factor in $\fint$, the change of variables formula turns~\eqref{eq:Jaco} into
$$\fint\limits_{B(g\inv x, 1/n)} f(gy) \left( \frac{\Vol({\scriptstyle B(g\inv x, 1/n)})}{\Vol({\scriptstyle gB(g\inv x, 1/n)})}|\Jac_g(y)| - 1 \right) dy.$$
Since $f$ is bounded, this converges to zero as $n\to\infty$ because of the characterisation of the Jacobian in terms of the volume of images of small balls. The convergence is uniform for $g$ in compact sets because of the $C^1$-continuity.

\medskip
We now verify that $\lambda$ is $G$-equivariant. Fix any $x\in X$ and $s\in G$. Given $i\in I$, the change of variables $g\to s\inv g$ yields for $(s \lambda_i \fhi)(x) = (\lambda_i \fhi) (s\inv x)$ the formula
$$(s \lambda_i \fhi)(x)  = \Lim_{n,\sU} \int\limits_G \fint\limits_{B(g\inv x, 1/n)} f(s\inv gy) \mu_i(s\inv g,s\inv x)  \,dy\, dg.$$
On the other hand,
$$(\lambda_i s \fhi)(x) = \Lim_{n,\sU} \int\limits_G \fint\limits_{B(g\inv x, 1/n)} f(s\inv gy) \mu_i(g,x)  \,dy\, dg.$$
Thus we have
$$\Big| (s \lambda_i \fhi)(x) - (\lambda_i s \fhi)(x) \Big| \leq \|\fhi\|_\infty \int\limits_G  \big|\mu_i(s\inv g,s\inv x) -\mu_i(g,x) \big|  \,dg.$$
This converges to zero by condition~\ref{pt:defamen:asy}. Therefore $\lambda$ is equivariant, finishing the proof.

\section{Proof of Theorem~\ref{thm:Lie}}
We keep the notation introduced for Theorem~\ref{thm:Lie}. Furthermore, we use $dg$, $dn$, $da$ and $dk$ to denote (integration against) a choice of left Haar measures on $G$, $N$, $A$ and $K$. When it causes no confusion, we simply denote by $|\cdot|$ the corresponding measure of measurable subsets. We normalise our choices in such a way that $|K|=1$ and that $dx$ is the projection of $dg$ to $X=G/K$. If we write the Iwasawa decomposition in the order $ANK$, then we can moreover assume $dg=da\,dn\,dk$, see~\cite[I.5.3]{Helgason00}. This now fixes all normalisations since we chose $dx$ to coincide with the Riemannian Lebesgue measure associated to the Killing form of $\mathfrak g$. Finally, we denote by $r=\dim\mathfrak a$ the $\RR$-rank of $G$.

We shall first investigate the sets $U_t\se X$ defined by $U_t= A_t \,D \,o$. We have $|U_t|= |A_t|\cdot|D|$ for the corresponding measures.

\begin{lem}\label{lem:A-inv}
We have
$$\lim_{t\to\infty} \frac{\big| a U_t \cap U_t \big|}{|U_t|} =1$$
uniformly for $a$ over compact subsets of $A$.
\end{lem}

\begin{proof}
It suffices to prove the statement with $A_t$ in place of $U_t$. By construction, there is a constant $c_1>0$ such that $|A_t| = c_1 t^r$. Therefore, it suffices to prove the following claim: there are $c_2, c_3 \geq 0$ such that for every $a\in A$ there is $b\in A$ with $\|\log b\| \leq c_2 \|\log a\|$ such that $a A_t \cap A_t$ contains $b A_{t-c_3 \|\log a\|}$ for all $t$. (By convention, $A_t$ is empty for $t<0$.) This claim, however, is an elementary property of simplicial cones in $\RR^r$ (one can take $b$ to be the exponential of a suitable multiple of the barycentric vector in $\mathfrak a^+$ dual to $\ro$).
\end{proof}

\begin{prop}\label{prop:N-inv}
We have
$$\lim_{t\to\infty} \frac{\big| u U_t \cap U_t \big|}{|U_t|} =1$$
uniformly for $u$ over compact subsets of $N$.
\end{prop}

\begin{proof}
Fix an arbitrary compact subset $C\se N$. By definition of the terms of the Iwasawa decomposition, $A^+$ contracts $N$ to the identity uniformly in the following precise sense. Given an identity neighbourhood $J$ in $N$ there exists $b\in A^+$ such that for all $a\in b A^+$ and all $u\in C$ we have $a\inv u a\in J$.

Let now $\epsilon>0$. Since $D$ is compact, there is an identity neighbourhood $J\se N$ such that $|JD|\leq (1+\epsilon)|D|$. Choose $b$ as above; by Lemma~\ref{lem:A-inv}, we have
\begin{equation}\label{eq:UA}
\big|U_t \cap b \,A^+ D\,o \big| \geq (1- \epsilon) |U_t|
\end{equation}
for all $t$ large enough. But whenever $x$ is in $U_t \cap b \,A^+ D\,o$ and $u\in C$, we have $ux\in A_t \,JD \,o$. Indeed, writing $x= a d o$ for $a\in A_t\cap b \,A^+$ and $d\in D$ yields $ux = a (a\inv u a)do$. Thus we deduce
$$u U_t \setminus U_t \ \se\ \big(A_t \,JD \,o \setminus U_t\big)  \cup u\big(U_t \setminus b \,A^+ D\,o \big),$$
and hence
$$\big|u U_t \setminus U_t \big| \leq \big| A_t \,JD \,o \setminus U_t \big| + \big|U_t \setminus b \,A^+ D\,o \big|.$$
The first term is less than $\epsilon |U_t|$ by the choice of $J$. For $t$ large enough independently of $u\in C$, the second terms is also less than $\epsilon |U_t|$, by~\eqref{eq:UA}. Hence, $|u U_t \cap U_t |\geq (1-2\epsilon) |U_t|$ and the proposition follows.
\end{proof}

We can now show that the sets $V^\xi_t$ of Theorem~\ref{thm:Lie} provide an explicit witness for the (well-known) amenability of the $G$-action on $B$, as follows.

\begin{prop}\label{prop:G-inv}
We have
$$\lim_{t\to\infty} \,\sup_{\xi\in B} \,\frac{\big |g V^\xi_t\cap V^{g\xi}_t \big|}{|V^\xi_t|} =1$$
uniformly for $g$ over compact subsets of $G$.
\end{prop}

\begin{proof}
We choose a map $B\to K$, $\xi\mapsto k_\xi$ such that $k_\xi M$ represents $\xi$ in $K/M$. We then define $p\colon G\times B\to G$ by $p(g,\xi) = k_{g\xi}\inv g k_{\xi}$ and observe that it ranges in a compact set whenever $g$ is restricted to a compact set. Moreover, $p(g,\xi)$ fixes the point of $B$ representing the trivial coset; in other words, $p$ ranges in $NAM$. In view of $V^\xi_t = k_\xi U_t$, we have
$$\big |g V^\xi_t\cap V^{g\xi}_t \big|\ =\ \big |p(g,\xi) U_t\cap U_t \big|.$$
Now the conclusion follows from Lemma~\ref{lem:A-inv} and Proposition~\ref{prop:N-inv}, recalling that $U_t$ is $M$-invariant.
\end{proof}

In order to conclude the proof of Theorem~\ref{thm:Lie}, it now suffices to show that the expression
$$(\lambda \fhi) (\xi)\ =\ \Lim_{n,\sU} \fint_{V_n^\xi} (P\fhi) (x) \,dx$$
defines a linear lifting $\lambda\colon L^\infty(B)\to \sL^\infty(B)$. Indeed, the $G$-equivariance follows from Proposition~\ref{prop:G-inv} exactly as in the proof of Theorem~\ref{thm:main} thanks to the $G$-equivariance of the Poisson transform $P$. We are going to use the Fatou theorem of Knapp--Williamson, Theorem~4.1 in~\cite{Knapp-Williamson}.

\smallskip
Lemma~\ref{lem:A-inv} implies that if $\tau\colon\RR_+\to\RR_+$ is a function tending sufficiently slowly to infinity, then
\begin{equation}\label{eq:slow}
\lim_{t\to\infty} \frac{\big| \exp(\tau(t) \alpha) U_t \cap U_t \big|}{|U_t|} =1
\end{equation}
holds uniformly for $\alpha$ over compact subsets of $\mathfrak a$. (In fact the proof of the lemma shows that it suffices to have $\tau(t)/t\to0$.) Choose once and for all an element $\alpha_0\in \mathfrak a^+$ and a function $\tau$ as above. Define the modified domains
$$\widetilde{V}_t^\xi = k_\xi \big(\exp(\tau(t) \alpha_0) U_t \cap U_t\big).$$
Then~\eqref{eq:slow} implies
$$\lim_{t\to\infty} \frac{\big| \widetilde{V}_t^\xi \cap V_t^\xi  \big|}{|V_t^\xi|} =1$$
and hence the following modified identity for $\lambda$
\begin{equation}\label{eq:modif:int}
(\lambda \fhi) (\xi)\ =\ \Lim_{n,\sU} \fint_{\widetilde{V}_n^\xi} (P\fhi) (x) \,dx
\end{equation}
holds. Since $\alpha_0$ is in the positive Weyl chamber, any sequence $x_n\in \widetilde{V}_n^\xi$ converges admissibly unrestrictedly to $\xi$ in the sense of Kor{\'a}nyi (Section~4 in~\cite{Koranyi69}). More precisely, the fact that $\alpha_0$ is regular implies that for any $T\in\mathfrak a$, our modified domain $\widetilde{V}_n^\xi$ lies within Kor{\'a}nyi's domain $\mathscr A^T_{D o}(\xi)$ in the notation of~\cite{Koranyi69} page~403 when $\tau(n)$ is large enough relative to $T$. Indeed,
$$\tau(n) \alpha_0 - T \in \mathfrak a^+ \ \Longrightarrow\ k_\xi (\exp(\tau(n) \alpha_0) A^+ Do \ \se\ \mathscr A^T_{D o}(\xi)$$
according to the definition of Kor{\'a}nyi's $\mathscr A^T_{D o}(\xi)$. Therefore, choosing $f\in\fhi$, the Knapp--Williamson Fatou theorem implies that for almost every $\xi\in B$, the Poisson transform $(P\fhi) (x)$ converges to $f(\xi)$ uniformly for $x$ in $\widetilde{V}_n^\xi$ as $n$ goes to infinity. In view of~\eqref{eq:modif:int}, this concludes the proof that $\lambda \fhi$ represents $\fhi$.

\smallskip
The fact that the lifting $\lambda$ is \emph{strong} can be justified as follows. The proof of Theorem~4.1 in~\cite{Knapp-Williamson} relies notably on Lebesgue-type strong differentiation much in the same was as the Fatou theorems of Fatou~\cite{Fatou}, Marcinkiewicz--Zygmund~\cite{Marcinkiewicz-Zygmund39}, etc.\ do, but differentiation has to be performed on the nilpotent group $\teta(N)$ instead of Euclidean differentiation ($\teta$ is the Cartan involution). Nonetheless, this provides convergence at least at every continuity point; compare~\S4 in~\cite{Koranyi69} and~\S4 in~\cite{Knapp-Williamson}.

\section{Proof of the remarks}\label{sec:rems}
Having in mind the proof of Theorem~\ref{thm:main}, we can address Remark~\ref{rem:cocycles}. Keep the notation of the theorem and let $p\in \NN$. We define $\mu_{p,i}$ on $G^p\times X^p$ as the product of $\mu_i$ and observe that it satisfies Properties~\ref{pt:defamen:norm} and~\ref{pt:defamen:asy} of Section~\ref{sec:main} for the $G^p$-action on $X^p$. Given a point $x=(x_1, \ldots, x_p)$ in $X^p$ and $r>0$, we denote by $C(x, r)$ the product of the balls $B(x_i, r)$. We claim that
\begin{equation}\label{eq:lift-amen-p}
(\lambda_p \fhi)(x)\ =\ \Lim_{i, \sI} \Lim_{n,\sU} \int\limits_{G^p} \fint\limits_{C(g\inv x, 1/n)} f(gy) \mu_{p,i}(g,x)  \,dy\, dg
\end{equation}
defines a $G^p$-equivariant strong linear lifting for $\fhi\in L^\infty(X^p)$. Indeed, the sets $C(x,r)$ have bounded eccentricity with respect to the balls in $X^p$ for the product metric. Thus the proof of Theorem~\ref{thm:main} applies unchanged to $\lambda_p$.

By construction $\lambda_p$ intertwines the permutation of coordinates. In order to check that the family $\lambda_p$ intertwines coboundaries, we recall that the homogeneous coboundary operator $d$ is the alternating sum  $\sum_{j=0}^p (-1)^j d_j$, where $d_j$ omits the $j$th coordinate. Thus it suffices to verify the following claim: let $f\in\mes(X^p)$ and define $d_0f\in\mes(x^{p+1})$ by $(d_0 f)(x) = f(x_1, \ldots, x_p)$, where $x=(x_0, x_1, \ldots, x_p)$. Then $\lambda_{p+1} [d_0 f] = d_0 \lambda_p [f]$. This is indeed apparent in the formula~\eqref{eq:lift-amen-p}.

\medskip
Remark~\ref{rem:cocycles} has an immediate application to cohomology. We recall that when considering cocycles given as function classes, a representative satisfying the cocycle equation everywhere is often called a \textbf{strict} cocycle to emphasize that the cocycle equation is not only assumed to hold amost-everywhere.

\begin{cor}\label{cor:coh}
Let $G$ be a locally compact group with a $C^1$-continuous amenable action on a differentiable manifold~$X$.

Then every continuous bounded cohomology class of $G$ (with real coefficients) can be represented isometrically by a $G$-invariant strict cocycle in $\mes(X^{p+1})$, where $p$ is the degree of the class.
\end{cor}

We refer to~\cite{Burger-Monod3,Monod} for the context of this result and for continuous bounded cohomology. The particular case where $G$ is a semi-simple Lie group is of special importance because explicit bounded cocycles for characteristic classes lead to numerical invariants in topology. The prime example of this phenomenon is given by Milnor--Wood inequalities~\cite{Milnor58, Wood71, Gromov}. We recall that \emph{all} characteristic classes are conjectured to be bounded~\cite{Dupont, MonodICM}; this is known to be the case for primary classes~\cite{Gromov, BucherKarlsson}.

\begin{proof}[Proof of Corollary~\ref{cor:coh}]
The $G$-action on $X$ is amenable also in the measure-theoretical sense of Zimmer, see~\cite[3.3.8]{Anantharaman-Renault}. Therefore, every cohomology class can be represented isometrically by an $L^\infty$-cocycle (function class) on $X^{p+1}$, see~\cite[Thm.~2]{Burger-Monod3} or~\cite[7.5.3]{Monod}. Now the result follows from Theorem~\ref{thm:main} and Remark~\ref{rem:cocycles}.
\end{proof}

\begin{rem}\label{rem:family-Lie}
The lifting of Theorem~\ref{thm:Lie} for a semi-simple Lie group $G$ also leads to a family of liftings for $G^p$ acting on $B^p$ that satisfy the additional properties of Remark~\ref{rem:cocycles}. Indeed it suffices to make consistent choices for the Weyl chamber etc.\ for $G^p$, taking $(A^+)^p \se A^p$ and the compact set $D^p$ in $N^p$. The resulting lifting is in fact equivariant for the almost connected semi-simple group $\mathrm{Sym}(p) \ltimes G^p$, and thus well-suited to implement Corollary~\ref{cor:coh}.
\end{rem}

Turning to Remark~\ref{rem:euler}, let $m\geq 4$ be an even integer and consider the group $G=\GL_m^+(\RR)$ of matrices with positive determinant acting on the projective space $\PP\RR^m$. In order to justify the remark, it suffices to show that there is a non-null bounded measurable $G$-invariant alternating function $f\colon  (\PP\RR^m)^{m+1}\to\RR$ such that the coboundary $df$ is a null-function, but such that for every $G$-invariant alternating function $f'$ in the class of $f$, the function $df'$ does not vanish everywhere on $(\PP\RR^m)^{m+2}$.

The argument is very similar to~\cite[\S3]{Bucher-Monod}, except that in this reference matrices with negative determinant were used to obtains cancellations; we will have to avoid this trick since we have no sign-equivariance here. In any case, the existence of the non-null $f$ with negligible coboundary $df$ is established therein. Let thus $f'$ be $G$-invariant, alternating and in the class of $f$; we shall verify that $df'$ is not zero. Denote by $e_1, \ldots, e_m\in\PP\RR^m$ the images of the usual basis vectors. Then there are exactly two $G$-orbits of $(m+1)$-tuples $(x, e_1, \ldots, e_m)$ such that $x$ has no zero entries, and $|f|$ is constant on the union of these two orbits, which is co-null (see~\cite[\S3]{Bucher-Monod}). Thus $f'$ cannot vanish at $(e_0, e_1, \ldots, e_m)$, where $e_0 $ is the class of the sum of all basis vectors. We shall prove the remark by showing that $df'$ does not vanish on the $(m+2)$-tuple $(e_0, e_1, \ldots, e_m, e_{1,2})$, where $e_{1,2}$ is the class of the sum of the first two basis vectors. To that end, it suffices to show that $d_jf'$ vanishes there for all $0\leq j\leq m$, since $d_{m+1}f'$ gives $f'(e_0, e_1, \ldots, e_m)$ which is non-zero.

For $j=0$, consider the block-diagonal matrix $g=\left(\begin{smallmatrix} 0&1\\ 1&0\end{smallmatrix}\right)\oplus -1 \oplus\Id_{m-3}$ which is in $G$. It permutes $e_1$ with $e_2$ but fixes $e_{1,2}$ and $e_i$ for all $i\geq 3$; therefore $f'$ vanishes on $(e_1, \ldots, e_m, e_{1,2})$.

For $j=1$, we consider $g=\left(\begin{smallmatrix}1&0&0&0\\ 0&1&0&0\\ 0&0&0&1\\ 2&0&-1&0\end{smallmatrix}\right)\oplus\Id_{m-4}$. It permutes $e_3$ with $e_4$ but fixes all other coordinates of $(e_0, e_2, \ldots, e_m, e_{1,2})$, so that $f'$ vanishes there. A similar matrix works for $j=2$.

For $j=3$, take $g=\left(\begin{smallmatrix}0&1\\ 1&0\end{smallmatrix}\right)\oplus \left(\begin{smallmatrix}1&0\\ 2&-1\end{smallmatrix}\right)\oplus\Id_{m-4}$. It permutes $e_1$ with $e_2$ but fixes all other coordinates of $(e_0, e,1, e_2, e_4, \ldots, e_m, e_{1,2})$, so that $f'$ vanishes there. A similar matrix works for $4\leq j\leq m$, finishing the proof.

\section{Remarks on the proofs}\label{sec:rems-p}
\noindent
(a)~Whilst the amenability of the action seems essential for the proof of Theorem~\ref{thm:main}, it is not a necessary condition. First of all, one could add a non-amenable factor to $G$ acting trivially on $X$. Another trivial case arises when $X$ is discrete. There are however more essential counter-examples, even going back to our basic guiding example of the projective action of $\SL_2(\RR)$ on $\RR\cup\{\infty\}$. We recall that the group $G(\RR)$ of \emph{piecewise} projective homeomorphisms of $\RR\cup\{\infty\}$ and its subgroup $H(\RR)$ of piecewise projective homeomorphisms of $\RR$ are both non-amenable~\cite{Monod_PNAS}. They contain many non-amenable finitely generated subgroups which, due to their self-similar nature, \emph{act non-amenably} on $\RR\cup\{\infty\}$, respectively on $\RR$; in fact their point stabilisers are non-amenable (compare~\cite{Monod_PNAS}). But by its local construction, the lifting of Theorem~\ref{thm:Lie} remains unaffected (compare also~(d) below).

\begin{prop}\label{prop:pw}
The linear lifting of Theorem~\ref{thm:Lie} is $G(\RR)$-equivariant.
\end{prop}

\begin{question}
Is there a notion of \emph{local} amenability of an action which could be necessary and sufficient for the existence of equivariant linear liftings?
\end{question}

A concrete test-case is the following.

\begin{problem}
Prove that there is no $\SL_3(\RR)$-equivariant lifting on the projective space $\PP\RR^3$.
\end{problem}

\medskip
\noindent
(b)~The proof of Theorem~\ref{thm:main} uses an averaging device \emph{before} completing Lebesgue differentiation. We have observed in the introduction that one cannot, it seems, reverse the order and average a given (non-equivariant) lifting in this way. Nevertheless, it is plausible that our proof can be combined with more general constructions of liftings. This would lead to a positive answer to the following.

\begin{question}\label{qu:amen}
Let $G$ be a locally compact group with a continuous amenable action on a locally compact space $X$. Is there a $G$-equivariant linear lifting on $X$?
\end{question}

\medskip
\noindent
(c)~We refer to~\cite[4.1]{Monod} for the notion of \emph{relatively injective} Banach modules.

\begin{question}
Let $G\acts X$ be an amenable action as in Theorem~\ref{thm:main} or Question~\ref{qu:amen}. Is the Banach $G$-module $\mes(X)$ relatively injective?
\end{question}

Several related function spaces on $X$ have been shown to be relatively injective~\cite[p.~3875]{Monod_exact}, but this always relied on duality; the proof given therein would for $\mes(X)$ run into the measurability issues mentionned in the introduction. As for the quotient $L^\infty(X)$, its relative injectivity actually characterizes the measure-theoretical amenability in the sense of Zimmer, see~\cite[Thm.~2]{Burger-Monod3} or~\cite[7.5.3]{Monod}. This Zimmer-amenability follows from topological amenability~\cite[3.3.8]{Anantharaman-Renault}.

\bigskip
\noindent
(d)~By construction, the linear lifting of Theorem~\ref{thm:Lie} is given by convolution against a rather explicit \emph{approximate identity}: the average over $V^\xi_n$ of the Poisson kernel. For instance, for $\SL_2(\RR)$ and $D$ as in~\S\ref{sec:thm:Lie}, we have
\begin{equation}\label{eq:noyau}
\lambda \fhi (\xi) = \Lim_{n,\sU} \int_{-\infty}^{+\infty} \fhi(x) M_n(x-\xi)\,dx \kern5mm(\fhi\in L^\infty(\RR), \xi\in\RR)
\end{equation}
where $dx$ is the Lebesgue measure and the kernel $M_t$ is given for $t>0$ by
$$M_t(x)=\frac{1}{2\pi t}\int_{e^{-t}}^1 \int_{-1}^1 \frac{du\ dv}{v^2 + (uv-x)^2},$$
which can be re-written with $A(x)= x\arctan(x) -\frac12 \log(1+x^2)$ as
$$M_t(x) = \frac{1}{2\pi t x}\Big(A(1+e^t x) - A(1-e^t x) + A(1-x) - A(1+x) \Big).$$
That~\eqref{eq:noyau} defines a $\SL_2(\RR)$-equivariant linear lifting is due to the fact that $M_t - \Jac_g\cdot M_t\circ g$ tends to zero in $L^1$-norm as $t\to\infty$ for all $g\in \SL_2(\RR)$. This, however, implies the corresponding statement for all $g$ in the \emph{piecewise} $\SL_2(\RR)$ group. Therefore, $\lambda$ is equivariant for this larger group.

\medskip
Since $M_t$ appears as a convolutor in~\eqref{eq:noyau}, its Fourier transform operates more directly. In can be written as exponential integral $\widehat{M_t}(\omega) = \frac1t \int_{e^{-t} |\omega|}^{|\omega|} \frac{e^{-s} \sin s}{s^2} \,ds$
%
%
(in non-unitary angular frequency $\omega$).

\subsection*{Acknowledgements}
I am gratedful to Mikael de la Salle for pointing out an inaccuracy in a previous version of the proof of Theorem~\ref{thm:main}.

\bibliographystyle{../BIB/amsalphaNM}
\bibliography{../BIB/ma_bib}
\end{document}